\documentclass[english]{article}
\usepackage[T1]{fontenc}
\usepackage[latin9]{inputenc}
\usepackage{textcomp}
\usepackage{amsthm}
\usepackage{amsmath}
\usepackage{amssymb}

\makeatletter

\newcommand{\lyxmathsym}[1]{\ifmmode\begingroup\def\b@ld{bold}
  \text{\ifx\math@version\b@ld\bfseries\fi#1}\endgroup\else#1\fi}

\theoremstyle{plain}
\newtheorem{thm}{\protect\theoremname}[section]
  \theoremstyle{plain}
  \newtheorem{prop}[thm]{\protect\propositionname}
  \theoremstyle{definition}
  \newtheorem{defn}[thm]{\protect\definitionname}
  \theoremstyle{remark}
  \newtheorem{rem}[thm]{\protect\remarkname}
  \theoremstyle{definition}
  \newtheorem{example1}[thm]{\protect\examplename}
  \theoremstyle{plain}
  \newtheorem{cor}[thm]{\protect\corollaryname}
  \theoremstyle{definition}
  \newtheorem{problem1}[thm]{\protect\problemname}

\usepackage{babel}

\usepackage{babel}

\makeatother

\usepackage{babel}
  \providecommand{\corollaryname}{Corollary}
  \providecommand{\definitionname}{Definition}
  \providecommand{\examplename}{Example}
  \providecommand{\problemname}{Problem}
  \providecommand{\propositionname}{Proposition}
  \providecommand{\remarkname}{Remark}
\providecommand{\theoremname}{Theorem}

\begin{document}

\title{Reflexive cones\footnote{The last two authors of this research have been co-financed by the European Union (European Social Fund -
ESF)and Greek national funds through the Operational Program
"Education and Lifelong Learning" of the National Strategic
Reference Framework (NSRF) - Research Funding Program: Heracleitus
II. Investing in Knowledge society through the European Social
Fund}}

\author{E. Casini%
\thanks{Dipartimento di Scienza e Alta Tecnologia, Università dell'Insubria,
via Valleggio, 11, 22100 Como, Italy (emanuele.casini@uninsubria.it).%
}\quad{} E. Miglierina%
\thanks{Dipartimento di Discipline Matematiche, Finanza Matematica ed Econometria,
Università Cattolica del Sacro Cuore, via Necchi 9, 20123 Milano,
Italy (enrico.miglierina@unicatt.it)%
}\quad{} I.A. Polyrakis%
\thanks{Department of Mathematics, National Technical University of Athens,
Zografou 157 80, Athens, Greece (ypoly@math.ntua.gr)%
}\quad{}F. Xanthos%
\thanks{Department of Mathematics, National Technical University of Athens,
Zografou 157 80, Athens, Greece (fxanthos@math.ntua.gr)%
}}
\maketitle

\begin{abstract}
Reflexive cones in Banach spaces are cones with weakly compact
intersection with the unit ball. In this paper we study the
structure of this class of cones. We investigate the relations
between the notion of reflexive cones and the properties of their
bases. This allows us to prove a characterization of reflexive cones
in term of the absence of a subcone isomorphic to the positive cone
of $\ell_{1}$. Moreover, the properties of some specific classes of
reflexive cones are investigated. Namely, we consider the reflexive
cones such that the intersection with the unit ball is norm compact,
those generated by a Schauder basis and the reflexive cones regarded
as ordering cones in  Banach spaces. Finally, it is worth to point
out that a characterization of reflexive spaces and also of the
Schur spaces by the properties of reflexive cones is given.

\bigskip


\bigskip

\end{abstract}

\section{Introduction}

The study of cones is central in many fields of pure and applied mathematics.
In Functional Analysis, the theory of partially ordered spaces and
Riesz spaces is based on the  properties of cones, how these properties
are related with the algebraic and topological properties of the spaces,
the structure of linear operators, etc. In Mathematical economics
(see \cite{Aliprantis-Brown-Burkinshaw,Alipratis et al}) the theory
of partially ordered spaces is used in General Equilibrium Theory
and in finance the lattice structure is necessary for the representation
of the different kinds of derivatives. Also the geometry of ordering
cones is crucial in the theory of vector optimization, see in \cite{Gopfert et al}
and the reference therein.

Our motivation for this article was to study the class of cones $P$
of a Banach spaces $X$ which coincide with their second dual cone
in $X^{**}$, i.e. $P=P^{00}$ and we had called these cones reflexive.

In the first steeps of our work, Section \ref{sec:3}, we proved that
$P$ is reflexive if and only if the intersection $B_{X}^{+}=P\cap B_{X}$
of $P$ with the unit ball $B_{X}$ of $X$ is weakly compact. Since
this property seems more accessible and more natural, we decided to
start by this property as definition of reflexive cones. Based on
the above characterization of reflexive cones we give, Theorem ~\ref{thm:reflexivity},
a characterization of reflexive spaces.

We remark that in Banach spaces, cones with weakly compact $B_{X}^{+}$
(reflexive cones in the present terminology) have been studied in
\cite{Polyrakis2008} and the results of this article are applied
in economic models. Also in \cite{Casini-Miglierina2010} structural
properties of cones related with this kind of cones are given.

In Section \ref{sec:4} we continue the study of \cite{Polyrakis2008}
and \cite{Casini-Miglierina2010} by studying the bases of reflexive
cones. The relationships with the existence of bounded and unbounded
base of a reflexive cone allow us to prove that a closed cone $P$
of a Banach space is reflexive if and only if $P$ does not contain
a closed cone isomorphic to the positive cone $\ell_{1}^{+}$ of $\ell_{1}$,
Theorem ~\ref{thm:reflexive not contain l1+}. Note that necessary
and sufficient conditions in order a closed cone of a Banach space
to be isomorphic to the positive cone of $\ell_{1}$ are given in
~\cite{Polyrakis1988}.

Moreover, it is worth pointing out that the existence of a basic sequence
inside a reflexive cone, Theorem ~\ref{thm:structure-bounded base}
and ~\ref{thm:reflexive not contain l1+}, depends on the existence
of a bounded or an unbounded base of the cone.

During our study we found interesting examples of cones with norm
compact positive part $B_{X}^{+}$ of the unit ball and we called
these cones strongly reflexive. Section \ref{sec:5} is devoted to
the study of this class of reflexive cones. (We start by mentioning
an old result of Klee, Theorem ~\ref{thm:Klee}.)

We prove a characterization of Schur spaces as the Banach spaces where
every reflexive cone with a bounded base is strongly reflexive, Theorem
~\ref{cor:Characterization Schur property}.

We give also different examples of reflexive cones. Especially if
$X$ is a Banach lattice with a positive Schauder basis then we prove
that a strongly reflexive cone $P\subseteq X_{+}$ exists so that
the subspace $Y=P-P$ generated by $P$ is dense in $X$ and we give
a method for the determination of this cone, Theorem \ref{thm:strongly-reflexive in Banach-lattice}.

Also in Example ~\ref{exa:L_1-strongly reflexive} a strongly reflexive
cone $P\subseteq L_{1}^{+}[0,1]$ and in Example~\ref{exa:L_1-reflexive}
a reflexive cone $P\subseteq L_{1}^{+}[0,1]$ are determined so that,
in both cases, the subspace $Y=P-P$ generated by $P$ is dense in
$L_{1}[0,1]$.

We close Section \ref{sec:5} by proving that any positive operator
from a Banach lattice into a Banach space ordered by a reflexive
(strongly reflexive ) pointed cone is weakly compact(compact), see
Theorem~\ref{l2}. Although the proof of this result is easy,
combined with our examples of reflexive cones can determine
different classes of compact operators.

In Section \ref{sec:6} we give a characterization of the positive
cone $P$ of a Schauder basis $\{x_{n}\}$ of a Banach space $X$, in
terms of the properties of the basis itself, in the same spirit of
the classical result of James, \cite{James}. We show,
Theorem~\ref{theorem0} that if the cone $P$ is reflexive, then
$\{x_{n}\}$ is boundedly complete on $P$ and in Theorem
~\ref{theorem1} we prove that if the basis $\{x_{n}\}$ is shrinking
and boundedly complete on $P$, then $P$ is reflexive. We show also,
Example~\ref{Jam}, that the assumption \textquotedbl{}the shrinking
basis $\{x_{n}\}$ is boundedly complete only on the cone
$P$\textquotedbl{}, is not enough to ensure the reflexivity of the
whole space $X$.

In Section \ref{sec:7} we suppose $X$ is a Banach space ordered
by a reflexive cone $P$ and we study order properties of $X$. First
we show that if $P$ is normal then $X$ is order complete, Theorem
\ref{Ded}. In the sequel we study the lattice property. By companying
the basic result of~\cite{Polyrakis1988}, where characterizations
of the positive cone $\ell_{1}^{+}$ of $\ell_{1}$ are given and
our result that $P$ as a reflexive cone does not contains $\ell_{1}^{+}$,
we prove that if $P$ has a bounded base, then $P$ cannot be a lattice
cone (for the exact assumptions see in Theorem \ref{latcone} and
its corollaries).

Finally, we recall that in the Choquet theory, cones $P$ in locally
convex Hausdorff spaces $X$ with a compact base are studied. As we
note in Remark~\ref{Choquet}, our results are to an other direction
and independent of the ones of this theory.

\section{Notations and preliminaries}

In this article we will denote by $X$ a Banach space, by $X^{\lyxmathsym{\textasteriskcentered}}$
the norm dual of $X$ and by
\[
B_{X}=\left\{ x\in X:\left\Vert x\right\Vert \le1\right\} ,
\]
 the closed unit ball of $X$.

A nonempty, convex subset $P$ of $X$ is a \textbf{cone} if $\lambda P\subseteq P$
for every real number $\lambda\geq0$. If in addition, $P\cap(-P)=\left\{ 0\right\} $,
$P$ is a \textbf{pointed} or a \textbf{proper} cone. Note that in
the literature, instead of the terms cone and pointed cone the names
wedge and cone are often used.

For any $A\subseteq X$ we denote by $\overline{A}$, the closure
of $A$, by int$A$ the interior of $A$ by ${\rm co}(A)$ the convex
hull and by $\overline{{\rm co}}(A)$ the closed convex hull of $A$.
Also we denote by cone$(A)$ ($\overline{{\rm cone}}(A)$) the smallest
cone, (closed cone) containing $A$. By $\overline{A}^{w}$ we denote
the closure of $A$ in the weak topology and by $\overline{A}^{w^{*}}$
the closure of $A$ in the weak-star topology, whenever $A$ is a
subset of a dual space. If the set $A$ is convex, we have
\[
{\rm cone}(A)=\left\{ \lambda a:a\in A,\lambda\geq0\right\} .
\]
 If the set $A$ is closed and bounded, it is easy to show that $cone(A)$
is closed.

Suppose that $P$ is a cone of $X$. $P\subseteq X$ induces the
partial ordering $\leq_{P}$ in $X$ so that $x\leq_{P}y$ if and only
if $y-x\in P$, for any $x,y\in X$; in the sequel, for the sake of
simplicity, we will use the symbol $\leq$ instead of $\leq_{P}$
whenever no confusion can arise. This order relation is
antisymmetric if and only if $P$ is pointed. If $x,y\in X$ with
$x\leq y$, the set $[x,\; y]=\{z\in X\;|\; x\leq z\leq y\}$ is the
\textbf{order interval} defined by $x,y$. If $P-P=X$ the cone $P$ is
\textbf{generating}. The cone $P$ gives an \textbf{open
decomposition of $X$} if there exists $\rho>0$ so that $\rho
B_{X}\subseteq B_{X}^{+}-B_{X}^{+},$ where $B_{X}^{+}=B_{X}\cap P$.
In  Banach spaces any closed and generating cone gives an open
decomposition, ~\cite{Jameson}, Theorem 3.5.2. The cone $P$ is
\textbf{normal} if there exists $c\in\mathbb{R}$ so that for any
$x,\; y\in X,\;0\leq x\leq y$ implies $\|x\|\leq\; c\;\|y\|$.

A linear functional
$f$ of $X$ is said \textbf{positive} (on $P$) if $f(x)\geq0$ for
each $x\in P$ and \textbf{strictly positive } (on $P$) if $f(x)>0$
for each $x\in P,x\neq0$. The set of the continuous positive functionals
on the cone $P$ is a cone named the \textbf{dual} (or \textbf{polar})
\textbf{cone} of $P$ and it is denoted by
\[
P^{0}=\{x^{*}\in X^{*}:x^{*}(x)\geq0\,\,\text{ for each }x\in P\}.
\]
 If a strictly positive linear functional exists, the cone $P$ is
pointed. A convex subset $B$ of ~$P$ is a \textbf{base for the
cone ~$P$} if for each $x\in P,x\neq0$ a unique real number $f(x)>0$
exists such that $\frac{x}{f(x)}\in B$. Then the function $f$ is
additive and positively homogeneous on $P$ and $f$ can be extended
to a linear functional on $P-P$ by the formula $f(x_{1}-x_{2})=f(x_{1})-f(x_{2}),x_{1},x_{2}\in P$,
and in the sequel this linear functional will always be extended to
a linear functional on $X$. So we have: \textit{$B$ is a base for
the cone $P$ if and only if a strictly positive (not necessarily
continuous) linear functional $f$ of $X$ exists so that, $B=\{x\in P\mid f(x)=1\}.$}
Then we say that \textbf{the base $B$ is defined by the functional
$f$} and we denote it by $B_{f}$. If $B$ is a base for the cone
$P$ with $0\notin\overline{B}$, then a continuous linear functional
$x^{*}\in X^{*}$ separating $\overline{B}$ and $0$ exists. Then
$x^{*}$ is strictly positive and, if $B$ is bounded, the base for
$P$ defined by $x^{*}$ is also bounded. Therefore we can summarize
these facts as follows.

\textit{The cone $P$ has a base defined by a continuous linear functional
$x^{*}$ of $X$ if and only if $P$ has a base $B$ with $0\notin\overline{B}$.
If moreover the base $B$ is bounded the base for $P$ defined by
$x^{*}$ is bounded}.

Moreover it holds the following well-known result.
\begin{prop}
\label{pro:Jameson-bounded base-interior polar cone}(\cite{Jameson},
Theorem 3.8.4). A cone $P$ of a normed space $Y$ has a bounded base
$B$ with $0\notin\overline{B}$ if and only if the dual cone $P^{0}$
of $P$ in $Y^{*}$ has interior points. Moreover for every $x^{*}\in{\rm int}P^{0}$,
the base $B_{x^{*}}$ of $P$ defined by $x^{*}$ is bounded.
\end{prop}
Now we recall a notion introduced in \cite{Casini-Miglierina2010}. A
cone $P$ is a \textbf{mixed based cone} if there exists a strictly
positive $x^{*}\in P^{0}$ which is not an interior point of $P^{0}$
and $intP^{0}\neq\emptyset$, or equivalently if $P$ has a bounded
and an unbounded base defined (the bases for $P$) by continuous
linear functionals.

Let $Y$ and $Z$ be two normed spaces. The cone $P\subseteq Y$ is
\textbf{isomorphic} to the cone $K\subseteq Z$ if there exists an
additive, positively homogeneous and one-to-one map $T$ of $P$ onto
$K$ such that $T$ and $T^{-1}$ are continuous in the induced topologies.
Then we also say that $T$ is an isomorphism of $P$ onto $K$ and
that $P$ is \textbf{embeddable} in $Z$.

We close this section by a known result, useful in this article. Recall
that if $\{x_{n}\}$ is a Schauder basis of $X$, then

\[
P=\{\sum_{i=1}^{\infty}\lambda_{i}x_{i}\;|\;\lambda_{i}\geq0,\;\;\text{for any}\; i\},
\]
 is \textbf{the positive cone of the basis $\{x_{n}\}$}. A sequence
$\{x_{n}\}$ of an ordered Banach space is a \textbf{positive basis}
of $X$ if it is a Schauder basis of $X$ and the positive cone $X_{+}$
of $X$ and the positive cone of the basis $\{x_{n}\}$ coincide.

The usual bases of the spaces $\ell_{p}$ , $1\leq p<\infty$ and
the space $c_{0}$ are simple examples of positive bases.

The next easy proposition has been used in~\cite{Polyrakis2004} and we present it here  because it  is
useful in our study. \begin{prop} \label{rem:positive cone of a basic sequence}
Let $X$ be an infinite dimensional Banach space with a normalized
Schauder basis $\{x_{n}\}$ and suppose that $\{x_{n}^{*}\}$ is the
sequence of the coefficient functionals of $\{x_{n}\}$. If $x^{*}=\sum_{i=1}^{\infty}\lambda_{i}x_{i}^{*}\in X^{*}$
with $\lambda_{i}>0$ for each $i$, then $x^{*}$ is strictly positive
on the positive cone $P$ of the basis $\{x_{n}\}$, and $x^{*}$
defines an unbounded base $B_{x^{*}}$ on $P$. \end{prop}

Indeed, it is easy to see that $\lim_{i\longrightarrow\infty}\lambda_{i}=0$
and that $\{\frac{x_{n}}{\lambda_{n}}\}$ is an unbounded sequence
of $B_{x^{*}}$.

\section{Reflexive cones and a characterization of reflexivity\label{sec:3}}

We start  with  the notion of reflexive cone.
\begin{defn}
 A cone $P$ of a  Banach space $X$ is \emph{reflexive} if
 the set $B_{X}^{+}=B_{X}\cap P$ is weakly compact. \end{defn}
\begin{rem}
\label{rem:(i)-(ii) basic properties of reflexive cones} The following
properties of a reflexive cone follow immediately from the definition:
\begin{enumerate}
\item A reflexive cone is always closed. Indeed, if $x_{n}\in P$ and $x_{n}\longrightarrow x$,
there exists $\rho\in\mathbb{R}_{+}$ so that $x_{n}\in\rho B_{X}^{+}$
for each $n$, therefore $x\in\rho B_{X}^{+}\subseteq P$.
\item Any closed cone of a reflexive space is reflexive. The converse does
not hold in general but, if a Banach space $X$ has a reflexive and
generating cone $P$, then $X$ is reflexive. Indeed, by
~\cite{Jameson}, Theorem 3.5.2, the cone $P$ gives an open
decomposition of $X$ and hence the unit ball $B_{X}$ is a weakly
compact set.
\end{enumerate}
\end{rem}
Now let us recall some standard notations. Let $X$ be a Banach
space, we denote by $J_{X}:X\longrightarrow X^{**}$ the natural
embedding of $X$ in $X^{**}$ . For the sake of simplicity, for any
$x\in X$ we denote by $\widehat{x}$ the image $J_{X}(x)$ of $x$ in
$X^{**}$ and for any subset $A\subseteq X$ we denote by
$\widehat{A}$ the set $J_{X}(A)\subseteq X^{**}$. Of course for any
$x^{*}\in X^{*}$, $\widehat{x^{*}}$ is the natural image
$J_{X^{*}}(x^{*})$ of $x^{*}$ in $X^{***}$ and for any $C\subseteq
X^{*}$, $\widehat{C}$ is the set $J_{X^{*}}(C)\subseteq X^{***}$.
Finally, given a subspace $V$ in $X$, we denote by $V^{\perp}$ the
annihilator $V^{\perp}=\{x^{*}\in X^{*}\;|\; x^{*}(x)=0,\;\text{for
any}\; x\in V\}$ of $V$ in $X^{*}$.

The next result shows that a reflexive cone exhibits the same behavior
as a reflexive space with respect to the second dual.
\begin{thm}
\label{thm:P=00003D00003D00003DP00}A closed cone $P$ of a Banach
space $X$ is reflexive if and only if
\[
\widehat{P}=P^{00}.
\]
\end{thm}
\begin{proof}
Assume that the cone $P\subseteq X$ is reflexive. Then, the set $\widehat{B_{X}^{+}}$
is a weak$^{*}$ compact set and therefore, for any $\alpha>0$, the
set
\[
\widehat{P}\cap\alpha B_{X^{**}}=\alpha\widehat{B_{X}^{+}}
\]
 is a weak$^{*}$ closed set. By Krein-Smulian Theorem, \cite{Megginson},
Theorem 2.7.11, $\widehat{P}$ is a weak$^{*}$ closed cone. We thus
get $P^{00}=\widehat{P}$ because as one can easily show  $P^{00}$ is
the weak$^{*}$ closure of $\widehat{P}$.

Now we prove the other implication. Let us suppose that the equality$P^{00}=\widehat{P}$
holds. By Banach-Alaoglu theorem, the set
\[
\widehat{P\cap B_{X}}=\widehat{P}\cap B_{X^{**}}=P^{00}\cap B_{X^{**}}
\]
 is weak$^{*}$ compact. Since the map $J_{X}$ is a weak-to-relative
weak$^{*}$ homeomorphism from $X$ onto $\widehat{X}$,
\cite{Megginson}, Proposition 2.6.24, the set $B_{X}^{+}$ is a
weakly compact set of $X$ and, therefore the cone $P$ is
reflexive.
\end{proof}
The next step toward the characterization of reflexive spaces is the
following result where we prove that the third dual of a reflexive cone $P$
of a Banach space $X$ can be decomposed in the same way as the third
dual space of $X$ is decomposed by the well known formula, \cite{Fetter-Gamboa de Buen},
Lemma I.12:
\begin{equation}
X^{***}=\widehat{X^{*}}\oplus(\widehat{X})^{\perp}.\label{dec}
\end{equation}

\begin{thm} \label{lem:decomposition P***} If $P$ is a reflexive
cone of a Banach space $X$, then
\[
P^{000}=\widehat{P^{0}}+(\widehat{X})^{\perp}.
\]
 Moreover, every $p^{***}\in P^{000}$ has a unique decomposition
$p^{***}=x^{***}+y^{***},$ where $x^{***}\in\widehat{P^{0}}$ and
$y^{***}\in(\widehat{X})^{\bot}$. \end{thm} \begin{proof}

We will show first that $P^{000}\subseteq\widehat{P^{0}}+(\widehat{X})^{\perp}$.
By formula (\ref{dec}), every $p^{***}\in P^{000}$ has a unique
decomposition
\[
p^{***}=x^{***}+y^{***},
\]
 where $x^{***}\in\widehat{X^{*}}$ and $y^{***}\in(\widehat{X})^{\bot}$
and suppose that $x^{***}=\widehat{x^{*}}$, where $x^{*}\in X^{*}$.
For every $p^{**}\in P^{00}$ we have
$$
0\leq p^{***}(p^{**})=x^{***}(p^{**})+y^{***}(p^{**}).
$$
 Since $P$ is reflexive we have $P^{00}=\widehat{P}$, therefore
there exist $p\in P$ so that $\widehat{p}=p^{**}$. So we have
\begin{equation}
0\leq p^{***}(p^{**})=x^{***}(\widehat{p})+y^{***}(\widehat{p})=x^{***}(\widehat{p})=x^{*}(p)\label{eq:1-decompositionP000}
\end{equation}
 for every $p\in P.$ So we have $x^{*}\in P^{0}$ and $x^{***}\in\widehat{P^{0}}$,
therefore
\[
P^{000}\subseteq\widehat{P^{0}}+(\widehat{X})^{\perp}.
\]
 For the converse suppose that $p^{*}\in P^{0}$ and $y^{***}\in(\widehat{X})^{\bot}.$
Since $P$ is reflexive, every $p^{**}\in P^{00}$ is of the form
$p^{**}=\widehat{p},$ where $p\in P$, therefore we have
\[
(\widehat{p^{*}}+y^{***})(p^{**})=(\widehat{p^{*}}+y^{***})(\widehat{p})=p^{*}(p)\geq0.
\]
 The last relation implies that
\[
\widehat{P^{0}}+(\widehat{X})^{\perp}\subseteq P^{000}
\]
 which completes the proof.  \end{proof}

\begin{thm}\label{thm:reflexivity}A Banach space $X$ is reflexive
if and only if there exists a closed cone $P$ of $X$ so that the
cones $P$ and $P^{0}$ are reflexive. \end{thm}

\begin{proof} If $X$ is reflexive the thesis follows immediately.
Now let us suppose that there exists a closed cone $P\subseteq X$
such that $P$ and $P^{0}$ are reflexive. By Theorem~\ref{lem:decomposition P***},
we have
\begin{equation}
P^{000}=\widehat{P^{0}}+(\widehat{X})^{\perp}.\label{eq:1-reflexivity}
\end{equation}
 Moreover, by Theorem \ref{thm:P=00003D00003D00003DP00} it holds
\begin{equation}
P^{000}=\widehat{P^{0}}.\label{eq:2-reflexivity}
\end{equation}
 Since the decomposition of every element of $P^{000}$ is unique,
the comparison between (\ref{eq:1-reflexivity}) and
(\ref{eq:2-reflexivity}) implies that
$(\widehat{X})^{\bot}=\left\{ 0\right\} .$ From this we conclude
that $X^{**}=\widehat{X}$, proves the theorem.  \end{proof}

Theorem \ref{thm:reflexivity} implies that in every non reflexive
Banach  space a reflexive cone cannot have a dual cone which is
reflexive. The following example shows such a situation.
\begin{example1}
\label{exa:RademacherI} Let $X=L_{1}\left(\left[0,1\right]\right)$,
$Y$ is the closed subspace of $X$ generated by the Rademacher functions
$\{r_{n}\}$, and let $P$ be the positive cone of $\{r_{n}\}$. Recall
that $\left\{ r_{n}\right\} $ is a basic sequence in $L_{1}\left(\left[0,1\right]\right),$
equivalent to the standard basis of $\ell_{2}$, therefore $Y$
is isomorphic to $\ell_{2}$ and the cone $P$ is reflexive.  By Theorem
\ref{thm:reflexivity}, the dual cone  $P^{0}$ of $P$ in $L_{\infty}\left(\left[0,1\right]\right)$ is not reflexive.
\end{example1}
We underline that the previous example exhibits a reflexive cone $P$
in a non reflexive space $X$ where
$\overline{P-P}$ is a reflexive subspace of $X$. This does not hold in
general as the next example shows. We underline also that in the next example, the
subspace $P-P$ generated by $P$ is dense in $X$.
\begin{example1} \label{exa:L_1-reflexive}Suppose that
$X=L_{1}\left([0,1]\right)$ and
\[
D=\{d\in L_{1}\left([0,1]\right)\;:\;0\leq d\leq
\mathbf{1},\;||d||\geq\frac{1}{2}\}.
\]
 where $\leq$ is the usual order of $L_{1}\left(\left[0,1\right]\right)$
and $\mathbf{1}\in L_{1}\left(\left[0,1\right]\right)$ is the
function identically equal to $1$. Then $D$ is a closed, convex and
bounded set, therefore the cone $P$ of $L_{1}\left([0,1]\right)$
generated by $D$ is closed. For every $x\in P\cap B_{X}$ we have
$x=\lambda d,\; d\in D$ with $\lambda=\frac{||x||}{||d||}\leq2$,
therefore $P\cap B_{X}$ is a closed and convex subset of the order
interval $[0,2\mathbf{1}]$. Now we recall that each order interval
of $L_{1}[0,1]$ is weakly compact because $X$ has order continuous
norm, \cite{Aliprantis-Burkinshaw}, Theorem 12.9. Therefore the set
$B_{X}^{+}=P\cap B_{X}$ is weakly compact and hence the cone $P$ is
reflexive. Moreover, we remark that $\overline{P-P}=X$. Indeed, if
$\{I_{i}\}$ is the sequence of subintervals
$I_{1}=[0,\frac{1}{2}),\;
I_{2}=[\frac{1}{2},1],I_{3}=[0,\frac{1}{4}),I_{4}=[\frac{1}{4},\frac{1}{2}),I_{5}=[\frac{1}{2},\frac{3}{4}),I_{6}=[\frac{3}{4},1],...$
of $[0,1]$, $I'_{i}$ is the complement of $I_{i}$ and ${\cal
{X}}_{I'_{i}}$ is the characteristic function of $I'_{i}$, then
${\cal {X}}_{I'_{i}}\in D$ for any $i$. Moreover every element of
the Haar basis of $L_{1}\left([0,1]\right)$ can be written as the
difference of two functions of the form ${\cal {X}}_{I'_{i}}$,
therefore $P-P$ is dense in $X$. Finally, it is easy to see that the
cone $P$ is normal and the basis for $P$ defined by the constant
function $\mathbf{1}$ is bounded. \end{example1}

\section{Bases of reflexive cones \label{sec:4}}

This section is devoted to the study of the reflexive cones that
admit a base defined by a continuous linear functional. This class
of reflexive cones is a large subset of the whole class of reflexive
cones. Nevertheless, there exist some reflexive cones that have not
a base defined by a continuous linear functional, as shown by the
following example. \begin{example1} Let us consider an uncountable
set $\Gamma$, then space $\ell_{2}(\Gamma)$ endowed with the
pointwise order, is a reflexive Banach lattice without strictly
positive, continuous  linear functionals.  Therefore the lattice
cone $\ell_{2}^{+}(\Gamma)$ is a reflexive cone without a base
defined by a continuous linear functional. Moreover, the same
behavior appears in the spaces $\ell_{p}(\Gamma)$ with $1<p<\infty$.
\end{example1} We begin by recalling the following \textbf{Dichotomy
Theorem} about the bases of cones. \begin{thm} \label{thm:dichotomy}(
\cite{Polyrakis2008}, Theorem 4). Suppose that $\langle X,Y\rangle$
is a dual system. If $X$ is a normed space, $P$ is a $\sigma(X,Y)$-closed
cone of $X$ so that the positive part $B_{X}^{+}=B_{X}\cap P$ of
the unit ball $B_{X}$ of $X$ is $\sigma(X,Y)$-compact, we have:
either every base for $P$ defined by a vector $y\in Y$ is bounded
or every such base for $P$ is unbounded. \end{thm}
In our setting, by
 Theorem \ref{thm:dichotomy} we have:

\begin{thm} \label{thm:relexive implies not mixed} Any
reflexive cone  of the Banach space  is not a mixed
based cone. \end{thm}

The converse of the above theorem does not hold because $c_{0}^{+}$
is not a mixed based cone, Proposition \ref{rem:positive cone of a basic sequence},
but $c_{0}^{+}$ is not reflexive.

We now provide a sufficient condition to ensure that a given cone
is reflexive, based on an assumption about the boundedness of the
bases of the cone.
 \begin{prop} \label{pro:bounded bases imply reflexivity}Let
$X$ be a Banach space ordered by the closed cone  $P$. If the
set $P^{0s}$ of strictly positive and continuous linear functionals
of $X$ is nonempty and for any $x^*\in P^{0s}$ the base $B_{x^{*}}$ for $P$ defined by $x^{*}$ is bounded,
then the cone $P$ is reflexive.\end{prop}
\begin{proof} Since every
base for $P$ defined by $x^{*}\in X^{*}$ is bounded we have that
$P^{0s}=int(P^{0})$. Hence, by Lemma 3.4 in \cite{Casini-Miglierina2010},
we have that the base $B_{x^{*}}$ for $P$ is weakly compact for
every $x^{*}\in X^{*}$. Now let us fix $x^{*}\in P^{0s}.$ Since
$B_{x^{*}}$ does not contain zero, there exists a positive real number
$\rho$ such that $\rho B_{X}\cap B_{x^{*}}=\emptyset$. It is easy
to check that the set
\[
\bigcup_{0\leq\alpha\leq1}\alpha B_{x^{*}},
\]
 is a weakly compact set which contains the closed set $\rho B_{X}\cap P$.
Therefore, the cone $P$ is reflexive.  \end{proof} We underline
that the converse of Proposition \ref{pro:bounded bases imply
reflexivity} does not hold. Indeed the reflexive cone
$\ell_{2}^{+}$ is such that every base $B_{x^{*}}$ for
$\ell_{2}^{+}$ is unbounded for every strictly positive linear
functional $x^{*}\in\ell_{2}$, Proposition~\ref{rem:positive cone
of a basic sequence}.

The results proved about the bases of reflexive cones allow us to
formulate a characterization of reflexive cones in the framework of
the theory of Banach spaces.

Before to state the theorem we recall a known result, that will play
a central role in the proof: \textit{Let $\left\{ x_{n}\right\} $
be a sequence in a Banach space $X$ which is not norm-convergent
to $0$. If }$\left\{ x_{n}\right\} $\textit{ is weakly Cauchy and
not weakly convergent, then $\left\{ x_{n}\right\} $ has a basic
subsequence}, \cite{Delabriere}, Theorem 1.1.10.

\begin{thm} \label{thm:reflexive not contain l1+}A closed cone $P$
of a Banach space $X$ is reflexive if and only if $P$ does not contain
a closed cone isomorphic to the positive cone of $\ell_{1}$. \end{thm}
\begin{proof} Let $P$ be reflexive. Suppose that $Q\subseteq P$
is a closed cone isomorphic to $\ell_{1}^{+}$. Then $Q$ as a subcone of $P$ is also
reflexive. By Theorem 4.5 in \cite{Casini-Miglierina2010}, $Q$ is
a mixed based cone which contradicts Theorem \ref{thm:relexive implies not mixed}.
Hence, $P$ does not contain a closed cone isomorphic to $\ell_{1}^{+}$.

To prove the other side of the equivalence, let us suppose that $P$
does not contain a closed cone isomorphic to $\ell_{1}^{+}$. Now,
on the contrary, suppose that $P$ is not a reflexive cone. Then $B_{X}^{+}=B_{X}\cap P$
is not a weakly compact set, therefore there exists a sequence $\left\{ x_{n}\right\} $
in $B_{X}^{+}$ which does not admit a weakly convergent subsequences.
Since $P$ does not contain a closed cone isomorphic to $\ell_{1}^{+}$,
$\ell_{1}$- Rosenthal Theorem, \cite{Diestel}, ensures that there
exists a weakly Cauchy subsequence of $\left\{ x_{n}\right\} $, which
we denote again by $\left\{ x_{n}\right\} $. By the result mentioned
just before this theorem, $\left\{ x_{n}\right\} $ has a basic subsequence
which we denote again by $\{x_{n}\}$, for the sake of simplicity.
The sequence $\left\{ x_{n}\right\} $ does not have a weakly convergent
subsequences, therefore $\left\{ x_{n}\right\} $ is not weakly convergent
to $0$. Hence, there exists $x^{*}\in X^{*}$ and a subsequence $\left\{ x_{n_{k}}\right\} $
of $\left\{ x_{n}\right\} $ such that $x^{*}\left(x_{n_{k}}\right)\geq1$
for each $k\in\mathbb{N}$. Therefore $\left\{ x_{n_{k}}\right\} $
is a basic sequence of $\ell_{+}$-type and the cone
\[
K=\left\{ p\in X:\: p=\sum_{k=1}^{\infty}\alpha_{k}x_{n_{k}};\:\alpha_{k}\geq0\:{\rm for\, every}\: k\in\mathbb{N}\right\} \subseteq P
\]
 generated by $\left\{ x_{n_{k}}\right\} $ is isomorphic to $\ell_{+}^{1}$,
\cite{Singer volI}, Theorem 10.2. This contradict the facts that
$P$ does not contain a cone isomorphic to $\ell_{1}^{+}$, and the
proof is complete.  \end{proof} The previous result shows that the
impossibility to embed the cone $\ell_{1}^{+}$ in a closed cone
$P$ of a Banach space is equivalent to $P$ be reflexive. Hence, it
is interesting to know whether a cone is isomorphic to
$\ell_{1}^{+}$ or not. A detailed study about this topic can be
found in \cite{Polyrakis1988}. Moreover, Theorem
\ref{thm:reflexive not contain l1+} yields an interesting
corollary that says that two isomorphic cone are reflexive
whenever one of them is reflexive.
\begin{cor} If the closed cones
 $P\subseteq X,\, Q\subseteq Y$  of the
Banach spaces $X,Y$ are isomorphic we have:  $P$ is reflexive if and only if $Q$ is reflexive.
\end{cor}
\begin{proof} Let $T$ be an isomorphism of $P$ onto $Q$.  If we suppose that   $P$  reflexive we have that $Q$ is also reflexive as follows: If   $Q$ is nonreflexive, then   $Q$ contains a closed cone $R$
isomorphic to $\ell_{1}^{+}$, therefore $T^{-1}(R)$ is a closed cone of
$P$ isomorphic to $\ell_{1}^{+}$, a contradiction.  \end{proof}
The following two results concern the inner structure of a
reflexive cone under the assumption that either bounded or
unbounded base defined by a continuous functional exists.

\begin{thm} \label{thm:structure-bounded base}Suppose that $P$
is a reflexive cone of a Banach space $X$. If $P$ has a bounded
base defined by $x^{*}\in X^{*}$, then $P$ does not contain a basic
sequence. \end{thm}

\begin{proof} Let $\left\{ x_{n}\right\} \subseteq P$ be a basic
sequence. Since the sequence $y_{n}=\frac{x_{n}}{x^{*}(x_{n})}$ is
a basic sequence with $x^{*}(y_{n})=1$ for each $n$, Theorem 10.2
in \cite{Singer volI} shows that $\{y_{n}\}$ is a basic sequence
of $\ell_{+}$-type, hence the cone
\[
K=\left\{ p\in X:\:
p=\sum_{n=1}^{\infty}\alpha_{n}y_{n};\:\alpha_{n}\geq0\:{\rm for\,\,
every}\: n\in\mathbb{N}\right\} \subseteq P
\]
 generated by $\{y_{n}\}$ is isomorphic to $\ell_{1}^{+}$ which   is a
contradiction. \end{proof}
\begin{thm} \label{thm:structure-unbounded base}Suppose that  $P$ is a reflexive
cone of  a Banach space $X$. If $P$ has an unbounded base defined
by $x^{*}\in X^{*}$, then $P$ contains a normalized basic sequence
$\{x_{n}\}$ which converges weakly to zero. \end{thm} \begin{proof}
By our assumption that $P$ has an unbounded base $B_{x^{*}}$ there exists  a sequence $\left\{ y_{n}\right\} $ such
that $y_{n}\in B_{x^{*}}$ for every $n$ and $||y_{n}||\longrightarrow\infty$.
Since $P$ is  reflexive, the normalized sequence $x_{n}=\frac{y_{n}}{||y_{n}||}$
has a weakly convergent subsequence which we denote again by $\{x_{n}\}$
for the sake of simplicity. We claim that $\{x_{n}\}$ converges weakly
to zero. Indeed, let $\bar{x}$ be the weak limit of $\{x_{n}\}$.
Then $\bar{x}\in P$ and
\[
x^{*}(\bar{x})=\lim_{n\rightarrow\infty}x^{*}(x_{n})=\lim_{n\rightarrow\infty}\frac{1}{||y_{n}||}=0,
\]
 hence we have $\bar{x}=0$ because $x^*$ is  strictly positive
on $P$. We conclude the proof by applying the well-known Bessaga-
Pelczynski Selection Principle (see, e.g., \cite{Diestel}), which
implies that $\{x_{n}\}$ has a basic subsequence.  \end{proof}

 Now
we examine, in Banach spaces, the relations between the existence of a reflexive cone with an unbounded
base and the existence of a reflexive subspace.  First we
 note that the existence of a reflexive subspace implies the existence of a reflexive cone with an unbounded
base  defined by a continuous linear functional. Indeed, if  $V$ is an infinite
dimensional reflexive subspace of $X$, then  the cone $P$ generated by a
basic sequence $\left\{ v_{n}\right\} \subseteq V$ is  reflexive
and by
Proposition~\ref{rem:positive cone of a basic sequence}, $P$ has  an unbounded base defined by a vector $x^{*}\in X^{*}$.
In the special case where  $X$ has an unconditional basis we prove below that the converse is also true. For this proof we use
Theorem \ref{thm:structure-unbounded base} and the next   result by Bessaga-Pelczynski:
\textit{If a Banach space $X$ has an unconditional basis, then every
normalized weakly null sequence of $X$ contains an unconditional
basic sequence,}~\cite{Megginson}, Theorem 4.3.19. \begin{thm} \label{thm:unconditional basis-reflexive subspace}Let
$X$ be a Banach space with an unconditional basis. If $P$ is a reflexive
cone with an unbounded base defined by a vector of $X^{*}$, then
$\overline{P-P}$ contains an infinite dimensional reflexive subspace.
\end{thm} \begin{proof} Suppose that $P$ is a reflexive cone, with
an unbounded base defined by $x^{*}\in X^{*}$. By
Theorem~\ref{thm:structure-unbounded base}, $P$ contains a
normalized weakly null sequence. By the Bessaga-Pelczynski theorem
mentioned above, $P$ contains an unconditional basic sequence
$\{x_{n}\}$. Let $K$ be the cone generated by $\{x_{n}\}$. It is
evident that $K\subseteq P$ is reflexive. By \cite{Singer volI},
Theorem 16.3, $K$ is generating in the closed subspace $Y$ of $X$
generated by $\{x_{n}\}$, i.e. $Y=\overline{{\rm span}}\left\{
x_{n}\right\} =K-K.$ Hence, the cone $K$ gives an open decomposition
in $Y$, ~\cite{Jameson}, Theorem 3.5.2, therefore $Y$ is reflexive
subspace of $\overline{P-P}$.
\end{proof} It is known that both $c_{0}$ and $\ell_{1}$ have an
unconditional basis but they do not contain an infinite dimensional
reflexive subspace. Therefore we have
\begin{cor} If $X=c_{0}$ or $X=\ell_{1}$, then $X$ does not contain
a reflexive cone $P$ with an unbounded base defined by a vector of
$X^{*}$. \end{cor}

\section{Strongly reflexive cones\label{sec:5}}

In this section we restrict our attention to a subclass of the set
of reflexive cones. Namely, we deal with those cones $P$ such that
the positive part of the unit ball $B_{X}^{+}$ is a compact  set.
 \begin{defn} A cone $P$ of a Banach space
$X$ is \emph{strongly reflexive}, if the set $B_{X}^{+}=B_{X}\cap P$
is norm compact. \end{defn}

\begin{rem} Suppose that $P$ is a strongly reflexive cone of an
infinite dimensional Banach space $X$. Then we have:
\begin{enumerate}
\item \textit{$P$ is not generating, i.e $X\neq P-P$.} Indeed, if we suppose
that $P$ is generating, by ~\cite{Jameson}, Theorem 3.5.2, we have
that $P$ gives an open decomposition of $X$, therefore the unit ball
$B_{X}$ is compact and $X$ is a finite dimensional space.
\item \textit{ $P$ does not have interior points}. Indeed, if $P$ has
an interior point ,then $P$ is generating.
\end{enumerate}
\end{rem} We start our study of strongly reflexive cones by recalling
a known result proved by Klee in \cite{Klee} (see also in \cite{Alfsen},
Theorem II.2.6). \begin{thm} \label{thm:Klee}A pointed cone $P$
in a locally convex Hausdorff space has a compact neighborhood of
zero, if and only if $P$ has a compact base. \end{thm} We remark
that the previous theorem can be adapted to our setting as follows:
\textit{A pointed cone $P$ in a Banach space is strongly reflexive
if and only if $P$ has a norm compact base}.

Now we state two elementary properties of strongly reflexive cones,
based on the above theorem of Klee.

\begin{prop} \label{pro:strongly reflexive - bounded bases+normal}If
$P$ is a strongly reflexive and pointed cone of a Banach space $X$,
then
\begin{enumerate}
\item \label{enu:1prop str.refl.-bounded bases+normal} $P$ has a base
defined by a vector of $X^{*}$ and any such a base for $P$ is compact;
\item \label{enu:2prop.str.refl.-bounded bases+normal}$P$ is normal.
\end{enumerate}
\end{prop} \begin{proof} Suppose that $P$ is a strongly reflexive
and pointed cone of $X$. Then by the above theorem, $P$ has a
compact base $B$. Hence there exists $x^{*}\in X^{*}$ that separates
$0$ and $B$ with $x^{*}(x)>0$ for every $x\in B$ and it is easy to
show that the base $B_{x^{*}}$ of $P$ defined by $x^{*}$ is closed
and bounded. Thus $B_{x^{*}}$ is compact because it is contained in
a positive multiple of $B_{X}^{+}$. Since $P$ cannot be a mixed
based cone, Theorem \ref{thm:relexive implies not mixed}, every base
for $P$ defined by a vector $X^{*}$ is also bounded. By
\cite{Jameson}, Proposition 3.8.2, $P$ is normal because $P$ has a
bounded base. (With respect to the notations used in \cite{Jameson},
we underline that well-based cone means cone with a  bounded base
$B$ so that $0\not\in \overline{B}$ and self-allied cone means
normal cone).
\end{proof}

In the next theorem we give \textbf{a method of construction of a
reflexive cone by a weakly convergent sequence}. By this, we have
that  any infinite dimensional Banach space $X$, contains  an
infinite dimensional reflexive cone with a bounded base defined by a
vector of $X^{*}$.
\begin{thm} \label{thm:cones reflexive and strongly reflexive with
bounded base}Let $\{x_{n}\}$ be a sequence of a Banach space $X$
which converges weakly to zero, $x_{0}\not\in\overline{{\rm
co}}\{x_{n}\}$ and $D=\overline{co}\{x_{n}\}-x_{0}$.
\\
 If $P={\rm cone}(D)$ is the cone of $X$ generated by
$D$, then  $P$ is a reflexive cone with a bounded base defined
by a vector of $X^{*}$. Moreover the following statement hold:
\begin{enumerate}
\item \label{enu:1 thm cones reflexive and strongly reflexive with bounded base}If
$\|x_{n}\|\geq\delta>0$ for each $n$, then $P$ is not strongly
reflexive;
\item \label{enu:2 thm cones reflexive and strongly reflexive with bounded base}If
$lim_{n}\|x_{n}\|=0$ then $P$ is strongly reflexive.
\end{enumerate}
\end{thm} \begin{proof} The set of elements of the sequence $\{x_{n}\}$
is relatively weakly compact and by the Krein-Smulian Theorem, \cite{Megginson},
Theorem 2.8.14, $\overline{{\rm co}}\{x_{n}\}$ is weakly compact.
Therefore $D$ is a convex, bounded and weakly compact set such that
$0\not\in D$. The cone $P$ generated by $D$ is closed because $D$
is closed and bounded. Let $0<\rho<m=\inf\{||y||\;|\; y\in D\}$.
Then
\[
P\cap\rho B_{X}\subseteq\bigcup_{0\leq\lambda\leq1}\lambda D,
\]
 therefore $P\cap\rho B_{X}$ is weakly compact and the cone $P$
is reflexive. Since $0\notin D,$ there exists $x^{*}\in X^{*}$ such
that $x^{*}$ separates $D$ and $0$ and $x^{*}(x)>0$ for every
$x\in D$. It is easy to check that $x^{*}$ defines a bounded base
for the cone $P$.

If we suppose that $\|x_{n}\|\geq\delta>0$ for each $n$, the set
$P\cap\rho B_{X}$ is not norm compact, because $\{x_{n}-x_{0}\}$
does not have a norm convergent subsequence. Therefore $P$ is not
strongly reflexive and statement \ref{enu:1 thm cones reflexive and
strongly reflexive with bounded base}.
is true.\\
 If we suppose that $lim_{n}\|x_{n}\|=0$, then in the above proof
the set $D$ is compact and we have that statement \ref{enu:2 thm
cones reflexive and strongly reflexive with bounded base}. is also
true.  \end{proof}

We give now a characterization of the Banach spaces $X$ with the
Schur property. Recall that $X$ has the \textbf{Schur property} if
every weakly convergent sequence of $X$ is norm convergent. If $X$
has the Schur property it is clear that any reflexive cone of $X$
is strongly reflexive, because every weakly compact subset of $X$
is norm compact. Therefore the interesting part of the next characterization
is the converse.

\begin{thm} \label{cor:Characterization Schur property} A Banach
space $X$ has the Schur property if and only if every reflexive cone
of $X$ with a bounded base defined by a vector of $X^*$, is strongly reflexive. \end{thm}
\begin{proof} For the converse suppose that $X$ does not have the
Schur property. Then $X$ has a normalized sequence $\{x_{n}\}$
with $x_{n}\overset{w}{\longrightarrow}0$ and by Theorem
\ref{thm:cones reflexive and strongly reflexive with bounded
base}, statement \ref{enu:1 thm cones reflexive and strongly
reflexive with bounded base}, we have a contradiction.
\end{proof}

Now we give below two examples of strongly reflexive cones. The first
is a general example. Specifically we give a way to build a strongly
reflexive cone in the positive cone $X_{+}$ of any Banach lattice
$X$ with a positive basis.

\begin{thm} \label{thm:strongly-reflexive in Banach-lattice} If
$X$ is an infinite dimensional Banach lattice with a positive Schauder
basis $\{e_{i}\}$, then $X_{+}$ contains a strongly reflexive cone
$P$ so that $\overline{P-P}=X$. \end{thm}

\begin{proof} Without loss of generality for our proof,  we can
suppose that $\left\{ e_{i}\right\} $ is a normalized basis of $X$.
Consider the cone
\[
P=\{x=\sum_{i=1}^{\infty}x_{i}e_{i}\;:\;0\leq x_{i}\leq\alpha x_{i-1},\text{\;\ for each }\; i>1\},
\]
 where $\alpha\in(0,1)$ is a fixed real number. It is straightforward
to see that $P$ is a closed subcone of $X_+$. Also it is easy to show that
\[
y=\sum_{i=1}^{\infty}\alpha^{i-1}e_{i}\in P.
\]

Moreover we can show that    for every $x\in P$ we have $x_{i}\leq\alpha^{i-1}x_{1}$
for every $i>1$ and also  that
\begin{equation}
P=\{x\in X\;|\; 0\leq x\leq x_{1}y\},\label{eq:0 strongly reflexive in Banach lattice}
\end{equation}
 where $\leq$ is the order of $X$. By the positivity of the basis
$\{e_{i}\}$, for every $x\in P$ we have
\begin{equation}
0\leq x_{1}e_{1}\leq x,\label{eq:1aa}
\end{equation}
 therefore
\begin{equation}
x_{1}\left\Vert e_{1}\right\Vert \leq||x||,\label{eq:2bb}
\end{equation}
 because $X$ is a Banach lattice. But $||e_{1}||=1$ because the
basis $\left\{ e_{i}\right\} $ is normalized and by (\ref{eq:0
strongly reflexive in Banach lattice}), (\ref{eq:1aa}) and
(\ref{eq:2bb}) we obtain
\begin{equation}
0\leq x\leq x_{1}y\leq\left\Vert x\right\Vert y,\label{eq:3 strongly reflexive in Banach space}
\end{equation}
 for every $x\in P$. From this we deduce that $P\cap B_{X}\subseteq[0,y]$.
By \cite{Singer volI}, Theorem 16.3, every Banach lattice with a
positive Schauder basis has compact order intervals, therefore the
order interval $[0,\; y]$ is  compact and the cone $P$ is strongly
reflexive. Also $P$ is pointed because it is contained in $X_{+}$.

We shall show that $\overline{P-P}=X$. We remark that $e_{1}\in P$
and also that
\[
e_{i}=\dfrac{1}{\alpha^{i-1}}\left(e_{1}+\alpha e_{2}+...+\alpha^{i-1}e_{i}\right)-\dfrac{1}{\alpha^{i-1}}\left(e_{1}+\alpha e_{2}+...+\alpha^{i-2}e_{i-1}\right)\in P-P,
\]
 for every $i\geq2$. Therefore $\overline{P-P}=X$. \end{proof}

By {Megginson}, Theorem 4.2.22,    we have: \textit{if $\{x_{n}\}$
is an unconditional basis of a Banach space $X$, then $X$, ordered
by the positive cone of the basis $\{x_{n}\}$ is a Banach lattice
(under an equivalent norm)}, therefore we have the following
corollary:

\begin{cor} \label{thm:strongly-reflexive in Banach-lattice,cor}
If $\{x_{n}\}$ is an unconditional basic sequence of a Banach space
$X$, then the positive cone of the basis $\{x_{n}\}$ contains a
strongly reflexive cone $P$ so that $\overline{P-P}=Y$, where $Y$
is the subspace generated by the basis $\{x_{n}\}$. \end{cor}

Now we give another example of a strongly reflexive cone $P$ of  the
 space $L_{1}\left[0,1\right]$ such that $P\subseteq L_{1}^{+}\left[0,1\right]$ and
$\overline{P-P}=L_{1}\left[0,1\right]$.

\begin{example1} \label{exa:L_1-strongly reflexive} Let $I_{1}=[0,1],\: I_{2}=[0,\frac{1}{2}),\: I_{3}=[\frac{1}{2},1],\: I_{4}=[0,\frac{1}{4}),\: I_{5}=[\frac{1}{4},\frac{1}{2}),\: I_{6}=[\frac{1}{2},\frac{3}{4}),\: I_{7}=[\frac{3}{4},1],...$
be a sequence of subintervals of $[0,1]$. \end{example1} Suppose
that $X=L_{1}\left(\left[0,1\right]\right)$ and
$T:\ell_{1}\longrightarrow X$ so that
\[
T(\xi)=\sum_{i=1}^{\infty}\xi_{i}{\cal {X}}_{I_{i}},
\]
 where ${\cal {X}}_{I_{i}}$ is the characteristic function of $I_{i}$
and $\xi=(\xi_{i})\in\ell_{1}$. Then $T$ is linear and continuous
with $||T(\xi)||\leq||\xi||$. Let us consider
$\eta=(1,\alpha,\alpha^{2},\alpha^{3},...)\in\ell_{1}^{+}$ where
$\alpha\in(0,1)$ and
\[
K=\{\xi=(\xi_{i})\in\ell_{1}^{+}\;|\;\xi_{i}\leq\xi_{1}\alpha^{i-1},\;\text{for each}\; i\}.\]
Then \[K=\{\xi=(\xi_{i})\in\ell_{1}^{+}\;|\; \xi\leq \xi_1\eta\},\]
 and $K$ is the strongly reflexive cone of Theorem~\ref{thm:strongly-reflexive in Banach-lattice} for $X=\ell_1$.
 Suppose that  $Q=T(K)$ and that  $P$ is the closure of $Q$ in
$L_{1}\left(\left[0,1\right]\right)$. For every $\xi \in K$ we have

\[
\xi_{1}{\cal {X}}_{I_{1}}\leq T(\xi)\leq \xi_1 T(\eta),
\]
 where $\leq$ is the usual order of $L_{1}\left(\left[0,1\right]\right)$.
Therefore we have
\[
\xi_1\leq||T(\xi)||\leq\xi_1||T(\eta)||.
\]
 Now let $x\in P$. Then there exists a sequence $\left\{ \xi^{n}\right\} \subseteq K$
such that $x=lim_{n\longrightarrow\infty}T(\xi^{n})$. Let
$M=2\left\Vert x\right\Vert $. Then we can choose the sequence
$\left\{ \xi^{n}\right\} $ such that $||T(\xi^{n})||\leq M$ for each
$n$. Therefore we have
\[
\xi_{1}^{n}\leq||T(\xi^{n})||\leq2||x||,
\]
 for each $n$. If we suppose that $||x||\leq1$ we have $\xi_{1}^{n}\leq2$,
therefore $\xi^{n}\leq2\eta$, hence
\[
T(\xi^{n})\in D=T([0,2\eta]),
\]
 for each $n$. Since the order interval $[0,2\eta]$ of $\ell_{1}$
is compact, we have that  $D$ is compact and $x\in D$. So we have that
$P\cap B_{X}\subseteq D$ is a compact set and the cone $P$ is strongly
reflexive. Finally, we remark that ${\cal {X}}_{I_{1}}=T(e_{1})$,
${\cal {X}}_{I_{2}}=\frac{1}{a}(T(e_{1}+ae_{2})-T(e_{1}))$ and continuing
we have that ${\cal {X}}_{I_{i}}\in P-P$ for any $i$. Therefore
$P-P$ is dense in $X$, because the Haar basis of $L_{1}\left[0,1\right]$
is consisting by differences of the functions ${\cal {X}}_{I_{i}}$,
see in \cite{Megginson}, Example 4.1.27.

We conclude this section dealing with some properties about positive
operators between ordered Banach spaces.

\begin{thm}\label{l2} Suppose that $E$ and $X$ are  Banach spaces ordered by
the closed and pointed  cones  $P$ and $Q$. If the cone $P$  gives an open decomposition in $E$ and $Q$ is a reflexive
(respectively, strongly reflexive) cone of $X$, then any positive, linear
operator $T:E\longrightarrow X$ is weakly compact (respectively,
compact).

\end{thm}

\begin{proof} Suppose that $T:E\longrightarrow X$ is positive, linear
operator. By ~\cite{Aliprantis-Tourky}, Theorem 2.32, $T$ is
continuous. The set $T(B_{E}^{+})\subseteq Q$,   is convex and
bounded, therefore relatively weakly compact, because $Q$ is
reflexive. By our assumptions, $W=B_{E}^{+}-B_{E}^{+}$ is a
neighborhood of zero. The set $\overline{T(B_{E}^{+})}$ is weakly
compact, hence the set
$\overline{T(W)}\subseteq\overline{T(B_{E}^{+})}-\overline{T(B_{E}^{+})}$
is weakly compact and therefore $T$ is weakly compact. The case of
compact operator is analogous.  \end{proof} If $E$ is a Banach
lattice, then $E_{+}$ gives an open decomposition in $E$,
therefore by the above theorem we have: \begin{cor}\label{l2} Any
positive, linear operator from a Banach lattice $E$ into a
 Banach space $X$ ordered by a pointed, reflexive (respectively, strongly
reflexive)  cone is weakly compact (respectively,
compact).

\end{cor}

\section{Reflexive cones and Schauder bases \label{sec:6}}

In this section we study the reflexivity of the positive cone of a
Schauder basis of $X$ in terms of the properties of the basis. Our
study has been inspired by the classical results of James in \cite{James}.
Suppose that $X$ is a Banach space with a Schauder basis $\{x_{n}\}$
and let $\{x_{n}^{*}\}$ be the sequence of the coefficient functionals
of $\{x_{n}\}$. Throughout the whole section we will denote by $P$,
the positive cone of the basis $\{x_{n}\}$, i.e. $P=\{x=\sum_{i=1}^{\infty}\lambda_{i}x_{i}\;|\;\lambda_{i}\geq0,\;\text{ for each}\; i\}$
and by $Y$ the closed subspace of $X^{*}$ generated by $\{x_{n}^{*}\}$. Note that $\{x_{n}^{*}\}$ is a basic sequence in $X^*$.
We will also denote by $Q=\{f=\sum_{i=1}^{\infty}\lambda_{i}x_{i}^{*}\;|\;\lambda_{i}\geq0,\;\text{ for each}\; i\}$,
the positive cone of the sequence $\{x_{n}^{*}\}$, by $Q^{0}$ the
dual cone of $Q$ in $Y^{*}$ and by $\Psi(x)$ the restriction of
$\widehat{x}$ to $Y$.

\emph{We say that $\{x_{n}\}$ is }\textbf{\emph{boundedly complete
on $P$}}\emph{ if for each sequence $\{a_{n}\}$ of nonnegative real
numbers $\sup_{n\in\mathbb{N}}\{||\sum_{i=1}^{n}a_{i}x_{i}||\}<+\infty$
implies that $\sum_{i=1}^{\infty}a_{i}x_{i}\in P$.} This notion has
been defined in \cite{PolyrakisXanthos2011} for the study of special
properties of cones. A simple example of a basis which is boundedly
complete on $P$ but not boundedly complete on the whole space $X$,
is the summing basis of $c_{0}$.

In the next proposition we give a characterization of Schauder bases
which are boundedly complete on the cone $P$.

\begin{prop}\label{00-1} $\{x_{n}\}$ is boundedly complete on $P$
if and only if $\Psi(P)=Q^{0}$. \end{prop}

\begin{proof} Suppose that $\{x_{n}\}$ is boundedly complete on
$P$. We shall show that for any $y^{**}\in Q^{0}$, $y^{**}=\Psi(x)$
for some $x\in P$.

First we will show that $x=\sum_{n=1}^{\infty}y^{**}(x_{n}^{*})x_{n}$
exists. This will imply that $y^{**}=\Psi(x)$ because $y^{**}(x_{i}^{*})=\widehat{x}(x_{i}^{*})$,
for each $i$.

For each $m\in\mathbb{N}$ and each $z^{*}=\sum_{n=1}^{\infty}a_{n}x_{n}^{*}\in Y$
we have
\[
|(\sum_{n=1}^{m}y^{**}(x_{n}^{*})\widehat{x_{n}})(z^{*})|=|y^{**}(\sum_{n=1}^{m}a_{n}x_{n}^{*})|\leq M^{*}||y^{**}||\,||z^{*}||,
\]
 where $M^{*}$ is the basis constant of the basic sequence $\{x_{n}^{*}\}$,
therefore
\[
||\sum_{n=1}^{m}y^{**}(x_{n}^{*})\widehat{x_{n}}||\leq M^{*}||y^{**}||
\]
 where $\sum_{n=1}^{m}y^{**}(x_{n}^{*})\widehat{x_{n}}$ is considered
as a functional of $Y$ and its norm is considered on $Y$. By ~\cite{Megginson},
Lemma 4.4.3, we have:
\[
||\sum_{n=1}^{m}y^{**}(x_{n}^{*})\widehat{x_{n}}||\leq MM^{*}||y^{**}||,
\]
 where $M$ is the basis constant of the basis $\{x_{n}\}$.

Therefore $||\sum_{n=1}^{m}y^{**}(x_{n}^{*})x_{n}||\leq MM^{*}||y^{**}||$
for each $m\in\mathbb{N}$, hence $x=\sum_{n=1}^{\infty}y^{**}(x_{n}^{*})x_{n}\in P$
because $\{x_{n}\}$ is boundedly complete on $P$. Therefore, $\Psi(x)=y^{**}$
and $\Psi(P)=Q^{0}$.

For the converse suppose that $\Psi(P)=Q^{0}$ and suppose also that
$\{a_{n}\}$ is a sequence of nonnegative real numbers with $||\sum_{i=1}^{n}a_{i}x_{i}||\leq M$
for each $n\in\mathbb{N}$. Then $s_{n}=\sum_{i=1}^{n}a_{i}\Psi(x_{i})$
is also bounded in $Y^{*}$. By Alaoglou's Theorem, there exists a
subnet $\{s_{n_{a}}\}_{a\in A}$ of $\{s_{n}\}$ such that
\[
s_{n_{a}}\xrightarrow{{w^{*}}}y^{*}\in Q^{0}=\Psi(P).
\]
 Therefore $y^{*}=\sum_{i=1}^{\infty}y^{*}(x_{i}^{*})\Psi(x_{i})$,
in the weak-star topology of $Y^{*}$. Since $\Psi(P)=Q^{0}$, there
exists $x\in P$ so that we have that $y^{*}=\Psi(x)$. Since $\{x_{n}\}$
is a basis of $X$ we have that $x=\sum_{i=1}^{\infty}x_{i}^{*}(x)x_{i}$.
For any $i$ we have
\[
x_{i}^{*}(x)=\Psi(x)(x_{i}^{*})=y^{*}(x_{i}^{*})=\lim s_{n_{a}}(x_{i}^{*})=a_{i},\;\text{ for each}\; i.
\]
 Hence $x=\sum_{i=1}^{\infty}a_{i}x_{i}$ and $\{x_{n}\}$ is boundedly
complete on $P$.  \end{proof}

The next results show the link between reflexivity of $P$ and the
boundedly completeness of the basis $\{x_{n}\}$ on the cone $P$.
\begin{thm}\label{theorem0} If  the positive cone $P$ of the
Schauder basis $\{x_{n}\}$ of the Banach space $X$ is reflexive,
then $\{x_{n}\}$ is boundedly complete on $P$. \end{thm}

\begin{proof} Suppose that $\{a_{n}\}$ is a sequence of nonnegative
real numbers and suppose that $s_{n}=\sum_{i=1}^{n}a_{i}x_{i}$
such that $||s_{n}||\leq M$ for each $n\in\mathbb{N}$. Since the
cone $P$ is reflexive, there exists a subsequence $\{s_{k_{n}}\}$
such that $s_{k_{n}}\xrightarrow{{w}}x$, where
$x=\sum_{i=1}^{\infty}x_{i}^{*}(x)x_{i}$. Then
$x_{i}^{*}(x)=\lim_{n\rightarrow\infty}x_{i}^{*}(s_{k_{n}})=a_{i}$,
for each $i$, therefore $x=\sum_{i=1}^{\infty}a_{i}x_{i}\in P$ and
$\{x_{n}\}$ is boundedly complete on $P$.  \end{proof}

Recall that the basis $\{x_{n}\}$ of $X$ is \textbf{shrinking} if
$\{x_{n}^{*}\}$ is a basis of $X^{*}$. It is known,~\cite{James},
that \textit{a Banach space $X$ with a Schauder basis $\{x_{n}\}$
is reflexive if and only if $\{x_{n}\}$ is shrinking and boundedly
complete (on~$X$).} In the next theorem, under the weaker assumption
that $\{x_{n}\}$ is boundedly complete on~$P$, we prove that the
cone $P$ is reflexive. The assumption that the shrinking basis $\{x_{n}\}$
is boundedly complete only on the cone $P$, is not enough to ensure
the reflexivity of the whole space $X$, see Example ~\ref{Jam}
below.

\begin{thm}\label{theorem1} If the Schauder basis $\{x_{n}\}$ of
the Banach space $X$ is shrinking and boundedly complete on~$P$,
then $P$ is reflexive. \end{thm}

\begin{proof} We will show that $\widehat{P}=P^{00}$. For any $x^{**}\in P^{00}$
we have that $x^{**}=\sum_{i=1}^{\infty}x^{**}(x_{i}^{*})\widehat{x_{i}}$,
in the weak-star topology of $X^{**}$, because the basis $\{x_{n}\}$
is shrinking. Therefore we have
\[
s_{n}=\sum_{i=1}^{n}x^{**}(x_{i}^{*})\widehat{x_{i}}\xrightarrow{{w^{*}}}x^{**}.
\]
 Hence there exists $M>0$, such that $||s_{n}||=||\sum_{i=1}^{n}x^{**}(x_{i}^{*})x_{i}||\leq M$,
for each $n$. Also $x^{**}(x_{i}^{*})\geq0$ for each $i$, because
$x^{**}\in P^{00}$, therefore
\[
x=\sum_{i=1}^{\infty}x^{**}(x_{i}^{*})x_{i}\in P,
\]
 because $\{x_{n}\}$ is boundedly complete on $P$ and by the uniqueness
of the weak-star limit we have $x^{**}=\widehat{x}\in\widehat{P}$,
therefore $P$ is reflexive.  \end{proof}

We end this section with an example of a nonreflexive space $X$ with
a shrinking Schauder basis which is boundedly complete on its positive
cone $P$ but not boundedly complete on the whole space $X$.

\begin{example1}\label{Jam} Let $J$ be the James space, i.e. the
space of all real sequences $\left\{ \alpha_{n}\right\} $ such that
$\lim\alpha_{n}=0$ and
\[
\sup\left\{ \sum_{j=1}^{n-1}\left(\alpha_{k_{j}}-\alpha_{k_{j+1}}\right)^{2}\right\} ^{1/2}<+\infty
\]
 where the supremum is taken over all $n\in\mathbb{N}$ and all finite
increasing sequences $k_{1}<k_{2}<\cdots<k_{n}$ in $\mathbb{N}$.
The norm $\left\Vert \left\{ \alpha_{n}\right\} \right\Vert _{J}$
is defined to be this supremum. It is well known that the standard
unit vector $\left\{ e_{n}\right\} $ in $J$ is a monotone shrinking
basis (see \cite{Fetter-Gamboa de Buen} for a general reference of
James space). Of course such a basis is not boundedly complete since
$J$ is not a reflexive space. It is easy to check that $\left\{ e_{n}\right\} $
is not boundedly complete on its positive cone.

Also by the formula $(a-b)^{2}\leq2a^{2}+2b^{2}$, for any $a,b\in\mathbb{R}$
we have that any real sequence $\{a_{i}\}$ of $\ell_{2}$ belongs
to $J$ with
\[
||\{a_{i}\}||_{J}\leq\big(2\sum_{i=1}^{\infty}a_{i}^{2}\big)^{\frac{1}{2}}.
\]

Let us consider the sequence $x_{n}=(-1)^{n+1}e_{n}$. It is clear
that $\left\{ x_{n}\right\} $ is again a shrinking basis of $J$.
We will show that $\left\{ x_{n}\right\} $ is boundedly complete
on its positive cone $P=\{x=\sum_{i=1}^{\infty}\lambda_{i}x_{i}\;|\;\lambda_{i}\geq0,\;\text{ for each}\; i\}$.

So we suppose that $\{\lambda_{n}\}$ is a sequence of nonnegative
real numbers so that $\sup_{n}\left\Vert \sum_{i=1}^{n}\lambda_{i}x_{i}\right\Vert _{J}\leq M$.

Suppose that $a_{i}=\lambda_{i}$ if $i$ is odd and $a_{i}=-\lambda_{i}$
if $i$ is even.

Then $\sup_{n}\left\Vert \sum_{j=1}^{n}\alpha_{j}e_{j}\right\Vert _{J}\leq M$.
Also
\[
(a_{k}-a_{k+1})^{2}=(|a_{k}|+|a_{k+1}|)^{2}\geq a_{k}^{2},
\]
 if $k$ is odd and

\[
(a_{k}-a_{k+1})^{2}=(-|a_{k}|-|a_{k+1}|)^{2}\geq a_{k}^{2},
\]
 if $k$ is even, therefore

\[
0\leq\sum_{k=1}^{n}a_{k}^{2}\leq\sum_{k=1}^{n}(a_{k}-a_{k+1})^{2}\leq M^{2}.
\]

Therefore the sequence $\left\{ \alpha_{n}\right\} \in\ell_{2}$, hence

\[
\sum_{i=1}^{\infty}\alpha_{i}e_{i}=\sum_{i=1}^{\infty}\lambda_{i}x_{i},
\]
 exists in $J$ and the basis $\{x_{n}\}$ is boundedly complete on
$P$. By Theorem \ref{theorem1}, we have that the cone $P$ is
reflexive. Finally note that $J=\overline{P-P}$.
\end{example1}

The previous example together with Theorem \ref{thm:unconditional basis-reflexive subspace}
suggest the following unanswered question.
\begin{problem1}
 If $P$ is a reflexive cone of a Banach space $X$ with an unbounded
base defined by a vector of $X^{*}$, does the subspace $\overline{P-P}$
contain a reflexive subspace?
\end{problem1}

\section{Spaces ordered by reflexive  cones\label{sec:7}}

We start this section by  proving   that $X$, ordered by a normal and reflexive cone is \textbf{Dedekind
complete}, i.e. any increasing net of $X$, bounded from above, has
a supremum. Recall that  a Dedekind complete ordered space is
not necessarily a lattice.

\begin{thm}\label{Ded} Any Banach space $X$, ordered by a reflexive
and normal cone $P$, is Dedekind complete. \end{thm} \begin{proof}
Suppose that $(x_{a})_{a\in A}$ is a increasing net of $X$, such
that $x_{a}\leq x$ for each $a\in A$. It is enough to show that
$\sup\{x_{a}\,\,|\,\, a\geq a_{0}\}$, for some fixed $a_{0}$,
exists. For each $a\geq a_{0}$ we have that $0\leq
x_{a}-x_{a_{0}}\leq x-x_{a_{0}}$, hence $||x_{a}-x_{a_{0}}||\leq
M$ for each $a\geq a_{0}$, because $P$ is normal. For any
$x^{*}\in P^{0}$, $x^{*}(x_{a}-x_{a_{0}})$ is increasing and
bounded by $M||x^{*}||$, therefore
$\lim_{a}x^{*}(x_{a}-x_{a_{0}})$ exists and we denote this limit
by $f(x^{*})$, i.e. $f(x^{*})=\lim x^{*}(x_{a}-x_{a_{0}})$. Note
that the cone $P^{0}$ is generating in $X^{*}$ because $P$ is
normal, ~\cite{Aliprantis-Tourky}, Theorem 2.26, therefore any
$x^{*}\in X^{*}$, has a decomposition of the form
$x^{*}=x_{1}^{*}-x_{2}^{*}$ and we define
$f(x^{*})=f(x_{1}^{*})-f(x_{2}^{*})$. It is easy to show that $f$
is well defined and that $f(x^{*})=\lim x^{*}(x_{a}-x_{a_{0}})$.
Then $f$ is a continuous linear functional of $X^{*}$ and $f\in
P^{00}$. By the reflexivity of $P$ we have that
$\widehat{P}=P^{00}$, therefore there exists $y\in P$ such that
$f(x^{*})=x^{*}(y)$ for any $x^{*}\in X^{*}$. It is easy to see
that $x_{a}-x_{a_{0}}\xrightarrow{{w}}y$. We shall show that
$\sup_{a\geq a_{0}}(x_{a}-x_{a_{0}})=y$. Indeed, for any $x^{*}\in
P^{0}$ we have that $x^{*}(x_{a}-x_{a_{0}})\leq x^{*}(y)$. Since
the cone $P$ is closed we have that $P$ coincides with the dual
cone $(P^{0})_{0}=\{x\in X\;|\; x^{*}(x)\geq0,\text{for any}\;
x^{*}\in P^{0}\}$, of $P^{0}$ in $X$, therefore we have that
$x_{a}-x_{a_{0}}\leq y$. For any $w\geq x_{a}-x_{a_{0}}$ we have
$x^{*}(x_{a}-x_{a_{0}})\leq x^{*}(w)$, therefore $x^{*}(y)\leq
x^{*}(w)$, for any $x^{*}\in P^{0}$, hence $y\leq w$ and the
theorem is true.  \end{proof}

In the following we study the lattice property of Banach spaces ordered by reflexive cones. So we will suppose that our cones are pointed (proper). Note that any cone with a base and also any normal cone of $X$  are pointed.
We start by some notations.

It is known, see in \cite{Aliprantis-Tourky}, Corollary 2.48, that
{\it any reflexive space $X$ ordered by a  normal and  generating
cone $P$ has the Riesz decomposition property (i.e. for each
$x,y,z\in P$ we have:
 $x\leq y+z\Longrightarrow x=x_{1}+x_{2}$, where  $x_{1},x_{2}\in P$ with
  $x_{1}\leq y,\, x_{2}\leq z$) if and only if $X$ is a lattice.}

 Since any space that contains a generating, reflexive cone is reflexive,
the following result is an immediate translation of the above result  in the language  of reflexive cones.
\begin{thm} \label{thm:normal generating reflexive  cone and lattice}
 A Banach space $X$ ordered by a normal, generating and reflexive cone
$P$ has the Riesz decomposition property if and only if $X$ is a
lattice. \end{thm}

Let $Y$ be  a normed space  ordered by the pointed cone $P$.
 A point $x_{0}\in P\backslash\{0\}$
is an \textbf{extremal point} of $P$  if for any
 $x\in X$ with $0\leq x\leq x_{0}$ implies $x=\lambda x_{0}$. If $Y$ is
a vector lattice and there exists a real constant $a>0$ so that $x,y\in Y,|x|\leq|y|$
implies $||x||\leq a||y||$, then $Y$ is a \textbf{locally solid
vector lattice }.  We continue by the notions of the continuous projection property. This
property was defined  in \cite{Polyrakis1988} for the study of the
extreme structure of the bases for cones.  We say that an
extremal point $x_{0}$ of $P$ \textbf{has (admits) a positive projection}
if there exists a linear, continuous, positive projection $P_{x_{0}}$
of $Y$ onto the one dimensional subspace $[x_{0}]$ generated by
$x_{0}$, so that $P_{x_{0}}(x)\leq x$ for any $x\in P$. We say that
$Y$ has the \textbf{continuous projection property} if it holds:
$x_{0}\in Ep(P)$ implies that $x_{0}$ admits a positive projection.
Also in \cite{Polyrakis1988} it is proved that if $Z=P-P$ and the
cone $P$ is closed, we have:

\textit{$(i)$ $Y$ has the continuous projection property if and
only if $Z$, ordered by the cone $P$, has the continuous projection
property and $(ii)$ if $Z$ is a locally solid vector lattice then
$Z$ and therefore also $Y$, have the continuous projection property.} So the continuous projection
property depends on the cone $P$, therefore we can also say that
the cone $P$ has the continuous projection property.

As it is noted in \cite{Polyrakis1988}, if $Y$ is a Banach space,
$P$ is pointed, closed and generating and $Y$ has the Riesz decomposition
property, then $Y$ has the continuous projection property, therefore
\textit{in a Banach space, ordered by a pointed, closed and generating cone,
the continuous projection property is weaker than the Riesz decomposition
and also than the lattice property}. \\

Note that in Example 4.1, of \cite{Polyrakis1988},  a reflexive,
generating and normal cone $P$ of a  reflexive space $X$ with a
bounded base which fails the Riesz decomposition property, is given.

Recall also Theorem 4.1 of \cite{Polyrakis1988} which is basic for
the study of the lattice propery of the reflexive cones: \textit{Let
$P$ be an infinite-dimensional closed  cone of a Banach space $X$.
If $P$ has the continuous projection property and $P$ has a closed,
bounded base $B$ with the Krein-Milman property, then $P$ is
isomorphic to the positive cone $\ell_{1}^{+}(\Gamma)$ of some
$\ell_{1}(\Gamma)$ space.}

\begin{thm}\label{latcone} If $P$ is a reflexive cone of a Banach
space $X$ with a closed, bounded base $B$, then $P$ does not contain
an infinite dimensional, closed cone $K$ with the continuous projection
property. \end{thm}

\begin{proof} First we remark that the base $B$ for $P$ has the
Krein-Milman property, i.e. each closed, convex and bounded subset
$C$ of $B$ is the closed convex hull of his extreme points.
Indeed, for any $C\subseteq B$, closed and convex, by the
Krein-Milman Theorem, we have
$C=\overline{co}^{w}ep(C)=\overline{co}ep(C)$, where ${\rm
ext}(C)$ denotes the set of the extreme points of $C$ and $\overline{co}^{w}ep(C)$, $\overline{co}ep(C)$, is the closed convex hull of $ep(C)$ in the weak, norm, topology of $X$.  If we
suppose that $K\subseteq P$ is an infinite dimensional, closed cone with the continuous
projection property, then by \cite{Polyrakis1988}, Theorem 4.1,
$K$ is isomorphic to the positive cone $\ell_{1}^{+}(\Gamma)$ of
some $\ell_{1}(\Gamma)$ space, therefore $\ell_{1}^{+}$ is
contained in $K$ because $\ell_{1}^{+}$ is contained in
$\ell_{1}^{+}(\Gamma)$. This is a contradiction because $K$ as a
reflexive cone cannot contain $\ell_{1}^{+}$, therefore $P$ does
not contain an infinite dimensional, closed cone $K$ with the
continuous projection property and the theorem is true.
\end{proof}

If in the above theorem we suppose that $K\subseteq P$ is a closed
cone and the space $Y=P-P$ is ordered by the cone $K$, we have:
If $Y$ is a locally solid vector lattice, then by the above remarks
we have that $P$ has the continuous projection property and by the
theorem we have a contradiction. So we have the following corollary:

\begin{cor}\label{ch1} If $P$ is a reflexive cone of a Banach space $X$ with
a closed, bounded base $B$, then $P$ does not contain an infinite
dimensional closed cone $K$ so that the space $Y=P-P$, ordered by
the cone $K$, is a locally solid vector lattice. \end{cor}

If in the above theorem we suppose that $X$ is a vector lattice,
we have again a contradiction because the cone $P$ has the continuous
projection property. So we have the corollary:

\begin{cor}\label{ch2} Any reflexive and generating cone of an infinite dimensional
Banach space $X$ with a bounded base cannot be a lattice cone. \end{cor}

\begin{rem}\label{Choquet} Note  that in the Choquet
theory, pointed cones $P$ in locally convex Hausdorff spaces $X$ with a
compact base are studied, see in the book of Alfsen, \cite{Alfsen}.
If a cone has a compact base, then it has a compact neighborhood of
zero. One of the main purposes of the Choquet theory is to give necessary
and sufficient conditions, so that $X$ ordered by the cone $P$ to
be a lattice. To this end a theory of representation of the vectors of
the base by measures supported \textquotedbl{}close\textquotedbl{}
to the extreme points of the base is developed. A necessary condition,
in order $X$ to be a lattice, is of course the cone $P$ to be generating.
If $X$ is a Banach space, then any cone $P$ with a compact base
is strongly reflexive and by the fact that $P$ is generating we have
that $X$ is finite dimensional, therefore the basic problem of the
Choquet theory as it is formulated in locally convex Hausdorff spaces,
is  pointless in Banach spaces because it concerns only finite dimensional spaces.

Moreover,
if we suppose that the cone $P$ has  a weakly compact base, then
$P$ is a reflexive cone with a bounded base. By Theorem
\ref{latcone} and its corollaries we have:
 If the cone $P$ is generating, then   $X$  cannot be a lattice, therefore in the case of a Banach space
ordered by a generating, closed cone with a weakly compact base the
answer to the basic problem of the Choquet theory is negative. A
negative answer is given  also by  Corollary ~\ref{ch1}, in the case
where the space $Y=P-P$ generated by $P$ is a locally solid vector
lattice.
 \end{rem}

\end{document}